\documentclass{amsart}
\usepackage{amssymb}

\usepackage{amsmath, amscd, amssymb, amsthm}

\usepackage{bbm}
\usepackage{latexsym}
\usepackage{amsfonts}
\usepackage{graphicx}
\usepackage{color}

\usepackage[
dvipdfm, colorlinks=true, linkcolor=black,breaklinks=true,
urlcolor=blue, citecolor=black
]{hyperref}%
\usepackage{geometry}
\newtheorem{theorem}{Theorem}[section]

\newtheorem{claim}[theorem]{Claim}

\newtheorem{proposition}[theorem]{Proposition}
\newtheorem{question}[theorem]{Question}

\newcommand{\be}{\begin{equation}}
\newcommand{\ee}{\end{equation}}
\newcommand{\bea}{\begin{eqnarray}}
\newcommand{\eea}{\end{eqnarray}}

\begin{document}

\title[On the almost one half holomorphic pinching]
{A note on the almost one half holomorphic pinching}

\author{Xiaodong Cao}
\thanks{Research partially supported by a
Simons Collaboration Grant}
\address{Xiaodong Cao. Department of Mathematics, 310 Malott Hall,
Cornell University, Ithaca, NY 14853-4201, USA.}
\email{xiaodongcao@cornell.edu}

\author{Bo Yang}
\thanks{Research partially supported by an AMS-Simons Travel
Grant}

\address{Bo Yang. Department of Mathematics, 310 Malott Hall,
Cornell University, Ithaca, NY 14853-4201, USA.}
\email{{boyang@math.cornell.edu} }

\begin{abstract}

Motivated by a previous work of Zheng and the second named author,
we study pinching constants of compact K\"ahler manifolds with
positive holomorphic sectional curvature. In particular we prove a
gap theorem following the work of Petersen and Tao on Riemannian
manifolds with almost quarter-pinched sectional curvature.

\end{abstract}

\date{This is the version which the authors
submitted to a journal for consideration for publication in June
2017. The reference has not been updated since then}

\maketitle

\section{The theorem}

Let $(M, J, g)$ be a complex manifold with a K\"ahler metric $g$,
one can define the \emph{holomorphic sectional curvature} ($H$) of
any $J$-invariant real $2$-plane $\pi=\operatorname{Span}\{X, JX\}$
by
\[
H(\pi)=\frac{R(X, JX, JX, X)}{||X||^4}.
\]
It is the Riemannian sectional curvature restricted on any
$J$-invariant real $2$-plane (p165 \cite{KN}). In terms of complex
coordinates, it is equivalent to write
\[
H(\pi)=\frac{R(V, \overline{V}, V, \overline{V})}{||V||^4}
\] where $V=X-\sqrt{-1}JX \in T^{1,0}(M)$.

In this note we study pinching constants of compact K\"ahler
manifolds with positive holomorphic sectional curvature ($H>0$). The
goal is to prove the following rigidity result on a K\"ahler
manifold with the almost one half pinching.

\begin{theorem} \label{almost half}
For any integer $n \geq 2$, there exists a positive constant
$\epsilon(n)$ such that any compact K\"ahler manifold with
$\frac{1}{2}-\epsilon(n) \leq H \leq 1$ of dimension $n$ is
biholomorphic to any of the following
\begin{enumerate}
\item  $\mathbb{CP}^{n}$,
\item  $\mathbb{CP}^{k} \times \mathbb{CP}^{n-k}$,
\item  An irreducible rank $2$ compact Hermitian symmetric space of
dimension $n$.
\end{enumerate}

\end{theorem}

Before we discuss the proof, let us review some background on
compact K\"ahler manifolds with $H>0$. The condition $H>0$ is less
understood and seems mysterious. For example, $H>0$ does not imply
positive Ricci curvature, though it leads to positive scalar
curvature. Essentially one has to work on a fourth order tensor from
the viewpoint of linear algebra, while usually the stronger notion
of holomorphic bisectional curvature leads to bilinear forms.

Naturally one may wonder if there is a characterization of such an
interesting class of K\"ahler manifolds. In particular, Yau
(\cite{Yau1} and \cite{Yau2}) asked if the positivity of holomorphic
sectional curvature can be used to characterize the rationality of
algebraic manifolds. For example, is such a manifold a rational
variety? There is much progress on K\"ahler surfaces with $H>0$. In
1975 Hitchin \cite{Hitchin} proved that any compact K\"ahler surface
with $H>0$ must be a rational surface and conversely he constructed
examples of such metrics on any Hirzebruch surface $M_{2,
k}=\mathbb{P}(H^{k}\oplus 1_{\mathbb{CP}^1})$. It remains an
interesting question to find out if K\"ahler metrics of $H>0$ exist
on other rational surfaces.

In higher dimensions, much less is known on $H>0$ except recent
important works of Heier-Wong (see \cite{HeierWong2015} for
example). One of their results states that any projective manifold
which admits a K\"ahler metric with $H>0$ must be rationally
connected. It could be possible that any K\"ahler manifold with
$H>0$ is in fact projective, again it is an open question. We also
remark that some generalization of Hitchin's construction of
K\"ahler metrics of $H>0$ in higher dimensions has been obtained in
\cite{AHZ}.

If Yau's conjecture is true, then how do we study the complexities
of rational varieties which admit K\"ahler metrics with $H>0$? A
naive thought is that the global and local holomorphic pinching
constants of $H$ should give a stratification among all such
rational varieties. Here the local holomorphic pinching constant of
a K\"ahler manifold $(M, J, g)$ of $H>0$ is the maximum of all
$\lambda \in (0,1]$ such that $0<\lambda H(\pi^{,}) \leq H(\pi) $
for any $J-$invariant real $2-$planes $\pi, \pi^{,} \subset T_p(M)$
at any $p \in M$, while the global holomorphic pinching constant is
the maximum of all $\lambda \in (0,1]$ such that there exists a
positive constant $C$ so that $\lambda C \leq H(p,\pi) \leq C$ holds
for any $p \in M$ and any $J$-invariant real 2-plane $\pi \subset
T_p(M)$. Obviously the global holomorphic pinching constant is no
larger than the local one, and there are examples of K\"ahler
metrics with different global and local holomorphic pinching
constants on Hirzebruch manifolds (\cite{YZ}).

In a previous work of Zheng and the second named author (\cite{YZ}),
we observe the following result, which follows from some pinching
equality on $H>0$ due to Berger \cite{Berger1960} and recent works
on nonnegative orthogonal bisectional curvature (\cite{ChenX},
\cite{GuZhang}, and \cite{Wilking}).

\begin{proposition} [\cite{YZ}]
\label{half pinching} Let $(M^n,g)$ be a compact K\"ahler manifold
with $0<\lambda \leq H \leq 1$ in the local sense for some constant
$\lambda$, then the following holds:

(1) If $\lambda>\frac{1}{2}$, then $M^n$ is bibolomorphic to
$\mathbb{CP}^{n}$.

(2) If $\lambda=\frac{1}{2}$, then $M^n$ satisfies one of the
following
\begin{enumerate}
\item $M^n$ is biholomorphic to $\mathbb{CP}^{n}$.
\item $M^n$ is holomorphically isometric to $\mathbb{CP}^k \times
\mathbb{CP}^{n-k}$ with a product metric of Fubini-Study metrics.
Moreover, each factor must have the same constant $H$.
\item $M^n$ is holomorphically isometric to an irreducible compact
Hermitian symmeric space of rank $2$ with its canonical
K\"ahler-Einstein metric.
\end{enumerate}
\end{proposition}

Let us remark that in the case that K\"ahler manifold in Proposition
\ref{half pinching} is projective and endowed with the induced
metric from the Fubini-Study metric of the ambient projective space,
a complete characterization of such a projective manifold and the
corresponding embedding has been proved by Ros \cite{Ros}.

Comparing with Proposition \ref{half pinching}, we may view Theorem
\ref{almost half} as a rigidity result on compact K\"ahler manifolds
with almost one half-pinched $H>0$. For example, Hirzebruch
manifolds can not admit K\"ahler metrics whose global pinching
constants are arbitrarily close to $\frac{1}{2}$.

It is very interesting to find the next threshold for holomorphic
pinching constants and prove some characterization of K\"ahler
manifolds with such a threshold pinching constant. Before making any
reasonable speculation, it is helpful to understand examples on such
holomorphic constants of some canonical K\"ahler metrics. In this
regard, the K\"ahler-Einstein metric on a irreducible compact
Hermitian symmetric space has its holomorphic pinching constant
exactly the reciprocal of its rank (\cite{Chen1977}). The
K\"ahler-Einstein metrics on many simply-connected compact
homogeneous K\"ahler manifolds (K\"ahler $C$-spaces) also have
$H>0$, and it seems very tedious to work with corresponding Lie
algebras carefully to determine these holomorphic pinching constants
except in lower dimensions. It was observed in \cite{YZ} that the
flag $3$-manifold, the only K\"ahler $C$-space in dimesion $3$ which
is not Hermitian symmetric, has $\frac{1}{4}$-holomorphic pinching
for its canonical K\"ahler-Einstein metric. Note that
Alvarez-Chaturvedi-Heier \cite{ACH} studied pinching constants of
Hitchin's examples of K\"ahler metrics with $H>0$ on a Hirzebruch
surface. However, it remains unknown what is the best pinching
constant among all K\"ahler metrics with $H>0$ on such a surface. We
refer the interested reader to \cite{ACH} and \cite{YZ} for more
discussions.

\section{The proof}

The proof is motivated by the work of Petersen-Tao \cite{PT} on
Riemannian manifolds with almost quarter-pinched sectional
curvature.

Assume for some complex dimension $n \geq 2$, there exists a
sequence of compact K\"ahler manifolds $(M_k, J_k, g_k)$ ($k \geq
1$) whose holomorphic sectional curvature satisfies
$\frac{1}{2}-\frac{1}{4k} \leq H(M_k, g_k) \leq 1$, and none of
$(M_k, J_k)$ is biholomorphic to any of the three listed in the
conclusion of Theorem \ref{almost half}. In the following steps,
$c(n)$, maybe different from line to line, are all constants which
only depend on $n$.

\textbf{Step 1:} (A uniform lower bound for the maximal existence
time of the K\"ahler-Ricci flow)

It is well-known (\cite{KN} for example) that bounds on holomorphic
sectional curvature lead to bounds on Riemannian sectional curvature
and the full curvature tensor. In particular, for any unit
orthogonal vectors $X$ and $Y$ we have
\begin{align*}
Sec(X, Y)=R(X, Y, Y, X)=&\frac{1}{8}\
\Big[3H(\frac{X+JY}{\sqrt{2}})+3H(\frac{X-JY}{\sqrt{2}})
\\&-H(\frac{X+Y}{\sqrt{2}})-H(\frac{X-Y}{\sqrt{2}})
-H(X)-H(Y) \Big]
\end{align*}

From the works of Hamilton and Shi (\cite{H1}, \cite{H2}, and
\cite{Shi}, and Cor 7.7 in \cite{ChowKnopf} for an exposition of
these results on compact manifolds), we conclude that for any $k
\geq 1$, there exists a constant $T(n)>0$ such that the
K\"ahler-Ricci flow $(M_k, J_k, g_k(t))$ with the initial metric
$g_k$ is well-defined on the time interval $[0, T(n)]$ for any $k
\geq 1$. Moreover, we have $|Rm (M_k, g_k(t))|_{g(t)} \leq c(n)$ for
some constant $c(n)$ and all $t \in [0, T(n)]$ and $k \geq 1$.

\vskip 3mm

\textbf{Step 2:} (An improved curvature bound on a smaller time
interval)

This step is due to Ilmanen, Shi, and Rong (Proposition 2.5 in
\cite{Rong}). Namely, there exists constants $\delta(n)<T(n)$ and
$c(n)$ such that for any $t \in [0, \delta(n)]$
\begin{align*}
\min_{p, V \subset T_p(M_k)}Sec(M_k, g_k(t), p, V)-c(n)t &\leq
Sec(M_k, g_k(t), p, P) \\
&\leq \max_{p, V \subset T_p(M_k)}Sec(M_k, g_k(t), p, V)+c(n)t.
\end{align*}

It is direct to see that a similar estimate holds for holomorphic
sectional curvature. Indeed, there exists $\delta(n)$ and $c(n)$
such that for any $t \in [0, \delta(n)]$
\begin{align*}
\frac{1}{2}-c(n)\, t \leq H(M_k, g_k(t)) \leq 1+c(n) \, t.
\end{align*}

\textbf{Step 3:} (An injective radius bound on $g_k(t_0)$ for some
fixed $t_0 \in [0, \delta]$)

We observe that Klingenberg's injectivity radius estimates on
even-dimensional Riemannian manifolds with positive sectional
curvature (Theorem 5.9 in \cite{CE} or p178 of \cite{Petersen} for
example) can be adapted to show that for any $t \in [0,
\delta_1(n)]$ such that $inj (M_k, g_k(t)) \geq c(n)$ for some
constant $c(n)>0$. Indeed it will follow from the claim below:
\begin{claim}
Let $(M^n, g)$ be a compact K\"ahler manifold with positive
holomorphic sectional curvature $H \geq \delta>0$ and $Sec \leq K$
where the constant $K>0$, then the injectivity radius ${inj} (M^n,
g) \geq c(K)$ for some constant $c(K)$.
\end{claim}
The proof of the above claim goes along as the proof of Theorem 5.9
in \cite{CE} except that we need to use the variational vector field
as $J \gamma'(t)$ where $\gamma(t)$ is a closed geodesic. This is
where we use $H>0$. In fact such a kind of estimate has been proved
in \cite{ChenTian} assuming positive bisectional curvature.

\vskip 2mm

\textbf{Step 4:} (An lower bound on orthogonal bisectional
curvatures of $(M_k, g_k(t)$)

This step is motivated by Peterson-Tao \cite{PT}, where they derived
a similar lower bound estimate for isotropic curvatures of almost
quarter-pinched Riemannian manifolds along Ricci flow.

\begin{claim}\label{Guzhang bounds}
There exists some constant $\delta_2(n)$ such that any $t \in [0,
\delta_2(n)]$, the orthogonal bisectional curvature of $(M_k,
g_k(t))$ has a lower bound $-\frac{1}{k} e^{c(n)t}$ for some
constant $c(n)>0$.
\end{claim}

\begin{proof}[Proof of Claim \ref{Guzhang bounds}]

Note that Berger's inequality \cite{Berger1960} (see also Lemma 2.5
in \cite{YZ} for an exposition) implies the orthogonal bisectional
curvature of $(M_k, g_k(0))$ is bounded from below by
$-\frac{1}{4k}$. Now the proof is based on a maximum principle
developed in \cite{H2}. In the setup of orthogonal bisectional
curvature, it is proved (in \cite{ChenX}, \cite{GuZhang}, and
\cite{Wilking}) that nonnegative orthogonal bisectional curvature is
preserved under K\"ahler-Ricci flow. For the sake of convenience, we
simply write $(M, J, g(t))$ where $t \in [0, T(n)]$ instead of the
sequence $(M_k, J_k, g_k (t))$.

Following \cite{H2}, one may use Uhlenbeck's trick. Consider a fixed
complex vector bundle $E \rightarrow M$ isomorphic to $TM
\rightarrow M$, with a suitable choice of bundle isomorphims
$\iota_t: E \rightarrow TM$, one obtain a fixed metric
$\iota_t^{\ast} g(t)$ on $E$. Now choose an unitary frame
$\{e_{\alpha}\}$ on $T^{1, 0}(E)$ which corresponds to an evolving
unitary frame on $TM$ via $\iota_t$, let $R_{\alpha
\overline{\alpha} \beta \overline{\beta}}$ denote $R(\iota_t^{\ast}
g(t), e_{\alpha}, \overline{e_{\alpha}}, e_{\beta},
\overline{e_{\beta}})$, the evolution equation of bisectional
curvature reads (\cite{GuZhang} for example)
\begin{equation}
\frac{\partial}{\partial t} R_{\alpha \overline{\alpha} \beta
\overline{\beta}}=\Delta_{g(t)} R_{\alpha \overline{\alpha} \beta
\overline{\beta}}+\sum_{\mu, \nu} (R_{\alpha \overline{\alpha} \mu
\overline{\nu}} R_{\beta \overline{\beta} \mu \overline{\nu}}
-|R_{\alpha \overline{\mu} \beta \overline{\nu}}|^2+|R_{\alpha
\overline{\beta} \mu \overline{\nu}}|^2)  \label{evolution}
\end{equation}
Now assume $m(t)=\min_{U \perp V \in T^{1, 0}(E)} R(U, \overline{U},
V, \overline{V})$, and assume $m(t_0)=R_{\alpha \overline{\alpha}
\beta \overline{\beta}}$ for some $t_0$ and some point $p \in M$.
Consider the first and the second variation of $R_{\alpha
\overline{\alpha} \beta \overline{\beta}}$, follow the proof of
Proposition 1.1 in \cite{GuZhang} and the curvature bounds in Step 1
we conclude from (\ref{evolution}) that $\frac{d^{-} m(t)}{d
t}|_{t=t_0} \geq c(n) m(t_0) $ whenever $m(t_0)<0$. Therefore there
exists some time interval $[0, \delta_2(n)]$ either $m(t) \geq 0$ or
$\frac{d^{+}(-m(t))}{d t} \leq c(n) (-m(t))$ if $m(t)<0$. Recall
$m(0) \geq -\frac{1}{4k}$, in any case we end up with $m(t) \geq
-\frac{1}{k} e^{c(n) t}$ for all $t \in [0, \delta_2(n)]$.
\end{proof}

\vskip 3mm

\textbf{Step 5:} (A contradiction after taking the limit of $(M_k, ,
J_k, g_k(t))$)

Let us consider $(M_k, J_k, g_k(t))$ where $t \in (0, \delta_2(n)]$
be a sequence of K\"ahler-Ricci flow, from previous steps we
conclude that there exist $\delta_3(n)$ and $c(n)$ such that
\begin{enumerate}
\item   $|Rm|_{g_k(t)} (M_k, g_k(t)) \leq c(n)$ for and $k \geq 1$
and $t \in [0, \delta_3(n)]$.

\item  $inj (M_k, g_k(t_0)) \geq \frac{1}{c(n)}$ for some
$t_0 \in [0, \delta_3(n)]$

\item $Ric (M_k, g_k(t_0)) \geq \frac{1}{c(n)}$ for any $k$, this follows
from Step 2 and 4.
\end{enumerate}

It follows from Hamilton's compactness theorem of Ricci flow
\cite{H3} that $(M_k, J_k, g_k(t))$ converges to a compact limiting
K\"ahler-Ricci flow $(M_{\infty}, J_{\infty}, g_{\infty}(t))$ where
$t \in (0, \delta_2(n)]$. It follows from Step 4 that
$g_{\infty}(t)$ has nonnegative orthogonal bisectional curvature and
$H(g_{\infty}(t))>0$ for any $0<t \leq t_0$. Note that all $M_k$ and
$M_{\infty}$ are simply-connected (\cite{Tsukamoto}), it follows
from \cite{ChenX}, \cite{GuZhang}, and \cite{Wilking} that
$(M_{\infty}, g_{\infty}(t_0))$ must be of the following form.
\begin{equation}
(\mathbb{CP}^{k_1}, g_{k_1}) \times \cdots \times
(\mathbb{CP}^{k_r}, g_{k_r}) \times (N^{l_1}, h_{l_1}) \times \cdots
(N^{k_r}, h_{l_s}).  \label{product list}
\end{equation}
Where each of $(\mathbb{CP}^{k_i}, g_{k_i})$ has nonnegative
bisectional curvature and each of $(N^{l_i}, h_{l_i})$ is a compact
irreducible Hermitian symmetric spaces of rank $\geq 2$ with its
canonical K\"ahler-Einstein metric. Now consider a time $t_1<t_0$
close to $t=0$, it follows from Step 2 that $g_{\infty}(t_1)$ is
close to $\frac{1}{2}$-holmorphic pinching and also have the same
decomposition as (\ref{product list}). Indeed the decomposition
(\ref{product list}) is reduced to exactly the list in the
conclusion of Proposition \ref{half pinching}. To see it one may
also apply the formula of piching constants of a product of metrics
with $H>0$ (\cite{ACH}). However $(M_k, J_k)$ is not biholomorphic
to any of the three three listed in the conclusion of Theorem
\ref{almost half}. Now we have a sequence of K\"ahler manifolds
$(M_{\infty}, \phi_k^{\ast} J_k, \phi_k^{\ast} g_k(t_1))$ converging
to $(M_{\infty}, J_{\infty}, g_{\infty}(t_1))$ where $\phi_k:
M_{\infty} \rightarrow M_k$ are the diffeomorphisms from Hamilton's
compactness theorem. This is a contradiction since any compact
Hermitian symmetric space is infinitesimally rigid, i.e.
$H^{1}(M_{\infty}, \Theta_{M_{\infty}})=0$ where $\Theta$ is the
sheaf of holomorphic vector fields on $M_{\infty}$ (see Bott
\cite{Bott}). This finishes the proof of Theorem \ref{almost half}.

\section{A remark}

Note that Proposition \ref{half pinching} works in the case of the
local one half pinching, therefore it seems natural to ask:

\begin{question}
Does Theorem \ref{almost half} hold if we replace the global almost
one half pinching to the local one?
\end{question}

Another optimistic hope is that $H>0$ is preserved along
K\"ahler-Ricci flow as long as the initial metric has a suitable
large holomorphic pinching constant. We refer to \cite{YZ} for more
discussions.

\end{document}